\documentclass{mrlart7}
  \def\volno{20}
  \def\yrno{2013}
  \def\issueno{6}
  \def\lpageno{1208} %temporary numbers
\usepackage{amsmath,amssymb,amsrefs,enumerate,esint,pstricks, pst-plot}

\newcommand{\s}{\mathcal{S}(\mathbb{R}^N)}
\newcommand{\M}{\mathfrak{M}}
\newcommand{\ep}{\varepsilon}
\newcommand{\D}{\mathcal{D}}

\newtheorem{thm}{Theorem}
\newtheorem{lemma}[thm]{Lemma}
\newtheorem{prop}[thm]{Proposition}

\begin{document}

\title[Pseudodifferential operators of mixed type]{Pseudodifferential operators of mixed type adapted to distributions of $k$-planes}
\author[Stein]{Elias M. Stein}
\address{Department of Mathematics, Princeton University, Princeton, NJ 08540, USA}
\email{stein@math.princeton.edu}
%\thanks{* Supported in part by the National Science Foundation, DMS-0901040.}
\author[Yung]{Po-Lam Yung}
\address{Mathematical Institute, University of Oxford, Oxford, OX2 6GG, UK}
\email{yung@maths.ox.ac.uk}%\address{Mathematical Institute, University of Oxford, OX2 6GG, UK}
%\thanks{** Supported in part by the National Science Foundation, DMS-1201474.}
\dedicatory{Dedicated to Alex Nagel on the occasion of his retirement \\ and \\ D.H. Phong on the occasion of his $60$th birthday.}

\begin{abstract}
We study the phenomena that arise when we combine the standard pseudodifferential operators with those operators that appear in the study of some sub-elliptic estimates, and on strongly pseudoconvex domains. The algebra of operators we introduce is geometrically invariant, and is adapted to a smooth distribution of tangent subspaces of constant rank. We isolate certain ideals in the algebra whose analysis is of particular interest.
\end{abstract}

\maketitle

    \setcounter{section}{-1}
\section{Introduction}

In this paper we study a class of pseudodifferential operators, given in terms of a smooth distribution $\mathcal{D}$ of subspaces of constant rank $k$ of the tangent space. The resulting class will be an algebra of pseudolocal operators, that is geometrically invariant. It will consist of a 2-parameter family $\Psi^{m,n}$, such that $\Psi^{m,0}$ contains the standard `isotropic' algebra of pseudodifferential operators of order $m$, and $\Psi^{0,n}$ contains the `non-isotropic' pseudodifferential operators of order $n$. Thus we think of our algebra as one that has mixed homogeneities, and understanding our algebra sheds light on how `isotropic' pseudodifferential operators compose with `non-isotropic' pseudodifferential operators that are related to sub-elliptic estimates and several complex variables. The results are first stated on $\mathbb{R}^N$, where the formulation is the cleanest. We will then adapt the results to the setting of smooth manifolds, and show that many natural operators in the context of strongly pseudoconvex domains are contained in our algebra.

There has been a lot of work on pseudodifferential operators that are `non-isotropic', especially in the context of subelliptic analysis and several complex variables; see e.g. work of Nagel-Stein \cite{MR549321}, Phong-Stein \cite{MR648484}, Taylor \cite{MR764508}, Beals-Greiner \cite{MR953082} and Ponge \cite{MR2417549} for some antecedents of our results. Those in turn can be traced back to some earlier work on `non-isotropic' singular integrals related to sub-elliptic estimates, which can be found in e.g. Folland-Stein \cite{MR0367477}, Rothschild-Stein \cite{MR0436223}. For a general treatment of relevant topics on singular integrals, see Street \cite{StreetBook}. 

Our class of operators of order $(0,0)$ form a subalgebra of the algebra of flag kernels, as was investigated in \cite{MR1312498}, \cite{MR1376298}, \cite{MR1818111} and \cite{MR2949616} by M\"uller, Nagel, Ricci, Stein and Wainger, and in \cite{MR2602167} by G{\l}owacki. In fact, say on the Heisenberg group $\mathbb{H}^n = \mathbb{C}^n \times \mathbb{R}$, there are two classes of flag kernels, adapted to the flags $\{0\} \subset \mathbb{C}^n \subset \mathbb{H}^n$ and $\{0\} \subset \mathbb{R} \subset \mathbb{H}^n$, and the intersection of these two algebras of flag kernels gives rise to an algebra of pseudolocal operators; c.f. Nagel-Ricci-Stein-Wainger \cite{NRSW2012}. Our algebra $\Psi^{m,n}$, when $m = n = 0$, is then a variable coefficient version of this algebra of pseudolocal singular integrals. In this way, \cite{NRSW2012} was the starting point to our present work.

Our main results can be summarized as follows. First, given a distribution $\mathcal{D}$ on $\mathbb{R}^N$, and two real numbers $m,n$, we will define a class of symbols $S^{m,n}$ of order $(m,n)$, and associate to it a class of pseudodifferential operators $\Psi^{m,n}$. Operators in such classes can be composed, so that the composition of an operator of order $(m_1,n_1)$ with one of order $(m_2,n_2)$ is an operator of order $(m,n)$, with $m = m_1 + m_2$, $n = n_1 + n_2$. Furthermore, $\Psi^{m,n}$ is closed under taking adjoints, and enjoys a diffeomorphism invariance. In addition, while the operators associated to $S^{0,0}$ fail to be weak-type $(1,1)$, they will preserve $L^p$ for $1 < p < \infty$. On the other hand, for any $\varepsilon > 0$, $\Psi^{\varepsilon,-2\varepsilon}$ and $\Psi^{-\varepsilon,\varepsilon}$ form two particularly significant ideals of $\Psi^{0,0}$: they map $L^1$ to weak-$L^1$, and preserve the isotropic and non-isotropic Lipschitz spaces of any positive order. We have also a rather surprising fractional integration result on $L^p$, $p > 1$, for operators of order $(m,n)$, when both $-1 < m < 0$ and $-(N-1) < n < 0$; it is stronger than the one predicted by the composition of an operator of order $(m,0)$ with one of order $(0,n)$. The same class of operators also map $L^1$ into weak $L^{1^*}$ for the best possible exponent $1^*$, if in addition $n \ne m(N-1)$. Many of these results depend on a characterization of the kernels of operators of class $\Psi^{m,n}$, in terms of differential inequalities and cancellation conditions, that are of interest in their own right.

We then apply our results to the Szeg\"o projection, as well as the Calder\'on operator, on the boundary of smoothly bounded strongly pseudoconvex domains in $\mathbb{C}^n$.

The full proofs of many of the results stated below are rather lengthy, so only sketchy indications of the main ideas are given here. The details, as well as some further extensions, will appear elsewhere. It is to be emphasized that our results are valid for distributions $\mathcal{D}$ where no curvature assumptions are required. For simplicity of exposition, we limit ourselves below to the case of distributions of tangent subspaces of codimension 1; the results hold more generally for distributions of tangent subspaces of higher codimension. 

\section{Geometric Preliminaries} \label{sect:prelim}

\subsection{Assumptions on $\mathcal{D}$} \label{subsect:assump}

Suppose we are given a smooth distribution $\mathcal{D}$ of tangent subspaces in $\mathbb{R}^N$. Assume $\mathcal{D}$ is of codimension 1, and is given by the nullspace of a 1-form whose coefficients have uniformly bounded derivatives. Then there exists a global frame $$X_1, \dots, X_N$$ of tangent vectors on $\mathbb{R}^N$, such that the first $N-1$ vectors form a basis of $\mathcal{D}$ at every point. Furthermore, one can pick
$$
X_i = \sum_{j=1}^N A_i^j(x) \frac{\partial}{\partial x^j}, \quad i = 1,\dots,N,
$$
such that all coefficients $A_i^j(x)$ are $C^{\infty}$ functions, with
$$
\|\partial_x^I A_i^j\|_{L^{\infty}} \leq C_I \quad \text{for all multiindices $I$},
$$ 
and
$$
|\det(A_i^j(x))| \geq c > 0 \quad \text{uniformly as $x$ varies over $\mathbb{R}^N$}.
$$ 
Below we will fix such a frame $X_1,\dots,X_N$, and use that to construct our class of symbols. It is to be pointed out, however, that our symbol class will ultimately depend only on the distribution $\mathcal{D}$, and not on the particular choice of $X_1,\dots,X_N$. We will often need the coefficients of the inverse of the matrix $(A_i^j)$; let's denote that inverse by $(B_j^k)$. In other words, $$\sum_{j=1}^N  A_i^j(x) B_j^k(x) = \delta_i^k \quad \text{for every $x$}.$$

\subsection{The role of two linear maps $M_x$ and $L_x$} 

One can define a variable coefficient linear map on the cotangent spaces, namely
$$
M_x \colon T_x^*(\mathbb{R}^N) \to T_x^*(\mathbb{R}^N), \quad x \in \mathbb{R}^N,
$$
so that
$$
M_x \xi = \sum_{i=1}^N \left( \sum_{j=1}^N A_i^j(x) \xi_j \right) dx^i \quad \text{if $\xi = \sum_{j=1}^N \xi_j dx^j$}.
$$
Then $M_x$ restricts to a map
$$
M_x \colon \text{Annihilator} (\mathcal{D}_x) \to \text{span}(dx^N) \subset T_x^*(\mathbb{R}^N).
$$
In fact, $M_x$ maps the dual frame $\theta^1, \dots, \theta^N$ of $X_1, \dots, X_N$ to the frame $dx^1, \dots, dx^N$.

One can also define a variable coefficient linear map on the tangent spaces, namely
$$
L_x \colon T_x(\mathbb{R}^N) \to T_x(\mathbb{R}^N), \quad x \in \mathbb{R}^N,
$$
so that
$$
L_x = (M_x^{-1})^t.
$$
(Here $(M^{-1})^t$ denotes the inverse transpose of a linear map $M$.) More explicitly, 
$$
L_x v = \sum_{i=1}^N \left( \sum_{j=1}^N B_j^i(x) v^j \right) \frac{\partial}{\partial x^i}, \quad \text{if $v = \sum_{j=1}^N v^j \frac{\partial}{\partial x^j}$.}
$$
It follows that $L_x$ restricts to a map
$$
L_x \colon \mathcal{D}_x \to \text{kernel}(dx^N) \subset T_x(\mathbb{R}^N).
$$

\subsection{A variable seminorm on the cotangent bundle}

We need a variable seminorm on the cotangent bundle of $\mathbb{R}^N$, defined by 
$$
\rho_x(\xi) = \left( \sum_{i=1}^{N-1} \left| (M_x \xi)_i \right|^2 \right)^{1/2} \quad \text{if $M_x \xi = \sum_{i=1}^N (M_x \xi)_i dx^i$}.
$$ 
(Note only the first $N-1$ components of $M_x \xi$ are involved.) In addition, we write 
$$
|\xi| = \left( \sum_{i=1}^{N} \left| \xi_i \right|^2 \right)^{1/2}, \quad \text{if $\xi = \sum_{i=1}^N \xi_i dx^i$},
$$ 
which we think of as the Euclidean norm of $\xi \in T_x \mathbb{R}^N$.

\subsection{An induced quasi-distance on $\mathbb{R}^N$} \label{subsect:quasidistance}

It is useful to consider the map
$$
\Theta_0 \colon \mathbb{R}^N \times \mathbb{R}^N \to \mathbb{R}^N,
$$
given formally by 
$$
\Theta_0(x,y) = L_x(x-y);
$$
more precisely, the $i$-th component of $\Theta_0(x,y)$ is given by
$$
\Theta_0(x,y)_i := \sum_{j=1}^N B_j^i(x) (x^j-y^j), \quad i = 1,\dots,N.
$$

For each $x, y \in \mathbb{R}^N$, in addition to the Euclidean distance $|x-y|$, we can now define a quasi-distance $d(x,y)$, such that
$$
d(x,y) = |x-y|+  \left| \Theta_0(x,y)_N \right|^{1/2}.
$$ 
This has the following properties:
$$
d(x,y) \simeq \sum_{i=1}^{N-1} \left| \Theta_0(x,y)_i \right| + \left| \Theta_0(x,y)_N \right|^{1/2} \quad \text{if $d(x,y) < 1$},
$$
and
$$
d(x,y) \simeq |x-y| \quad \text{if $d(x,y) \geq 1$}.
$$
Also,
\begin{enumerate}[(a)]
\item $d(x,y) \simeq d(y,x)$ for all $x, y \in \mathbb{R}^N$, and
\item $d(x,z) \lesssim d(x,y) + d(y,z)$ for all $x, y, z \in \mathbb{R}^N$.
\end{enumerate}

This quasi-distance, together with the Lebesgue measure on $\mathbb{R}^N$, gives $\mathbb{R}^N$ the structure of a space of homogeneous type.

\subsection{Variants of $\Theta_0$} \label{subsect:Theta}

Suppose now $\Delta = \{(x,y) \in \mathbb{R}^N \times \mathbb{R}^N \colon |x-y| < \delta\}$, for some $\delta > 0$, is a tubular neighborhood of the diagonal of $\mathbb{R}^N \times \mathbb{R}^N$, and $\Theta$ is a $C^{\infty}$ map 
$$
\Theta \colon \Delta \to \mathbb{R}^N,
$$
with
$$
\|\partial_x^I \partial_y^J \Theta(x,y)\|_{L^{\infty}} \leq C_{I,J} \quad \text{for all multiindices $I$ and $J$}.
$$
We write 
$$
\mathcal{B}_j^k(x) = - \left. \frac{\partial}{\partial y^j} \right|_{y=x} \Theta(x,y)_k, 
$$
and assume in addition the existence of some absolute constant $c_0$ such that 
$$
|\det(\mathcal{B}_j^k(x))| \geq c_0 > 0 \quad \text{uniformly for all $x$}.
$$
Given such a map $\Theta$, we say that it is \textbf{compatible} with our distribution $\mathcal{D}$, if and only if 
\begin{equation} \label{eq:compat1}
\mathcal{D}_x = \text{the kernel of $\sum_{j=1}^N \mathcal{B}_j^N(x) dx^j$ for all $x$}.
\end{equation}
\newpage
\noindent Equivalently, it is compatible with $\mathcal{D}$, if the map 
$$
\mathcal{L}_x \colon T_x (\mathbb{R}^N) \to T_x (\mathbb{R}^N),
$$
$$
\mathcal{L}_x v := \sum_{i=1}^N \left( \sum_{j=1}^N \mathcal{B}_j^i(x) v^j \right) \frac{\partial}{\partial x^i} \quad \text{if $v = \sum_{j=1}^N v^j \frac{\partial}{\partial x^j}$}
$$
restricts to a map 
\begin{equation} \label{eq:compat2}
\mathcal{L}_x \colon \mathcal{D}_x \to \text{kernel}(dx^N) \subset T_x(\mathbb{R}^N)
\end{equation} for all $x$. The $\Theta_0$ defined in Section~\ref{subsect:quasidistance} is an example of such $\Theta$.

In applications, we will often work locally, and be given a distribution $\mathcal{D}$ of tangent subspaces on an Euclidean ball $B$, say $B = B(0,\delta/2)$. Given a smooth map $\Theta \colon B \times B \to \mathbb{R}^N$, we then say $\Theta$ is compatible with $\mathcal{D}$ on $B$, if (\ref{eq:compat1}) (or (\ref{eq:compat2})) holds for all $x \in B$. In this situation, one can show (c.f. Lemma 5.2 of \cite{MR0116352}) that there exists an absolute constant $\delta_0$ (depending only on finitely many of the structural constants $C_{I,J}$ and $c_0$) such that the following holds:
\begin{enumerate}[(i)]
\item there exists $\tilde{\Theta} \colon \Delta \to \mathbb{R}^N$, and a distribution $\tilde{\mathcal{D}}$ on $\mathbb{R}^N$, with $$\tilde{\Theta} = \Theta \quad \text{on $B(0,\delta_0) \times B(0,\delta_0)$,} \quad \text{ and } \quad \tilde{\mathcal{D}} = \mathcal{D} \quad \text{on $B(0,\delta_0)$};$$ 
\item $\tilde{\Theta}$ is compatible with $\tilde{\mathcal{D}}$ on the whole $\mathbb{R}^N$; and
\item $\tilde{\Theta}(x,y) = \mathcal{L}_0(x-y)$ if both $x, y \notin B(0,2\delta_0)$, and $\tilde{\mathcal{D}}_x = \mathcal{D}_0$ for $x \notin B(0,2\delta_0)$.
\end{enumerate}
Since we can restrict our attention to functions supported inside $B(0,\delta_0)$, we may, by abuse of notation, write $\Theta$ in place of $\tilde{\Theta}$, and $\mathcal{D}$ in place of $\tilde{\mathcal{D}}$, and we are back to our earlier set-up.

For example, one can define such a compatible $\Theta$ locally by the exponential map
\begin{equation} \label{eq:Thetaexp}
x = y \exp(\Theta(x,y) \cdot X)
\end{equation}
for $x,y$ on a ball $B$.

\subsection{Some special derivatives on the cotangent bundle} \label{subsect:goodcotder}

We need these to describe the symbol classes we have.

$D_{\xi}$ is a good derivative on the cotangent bundle of $\mathbb{R}^N$. To describe that, let $x$ be the standard coordinate system on $\mathbb{R}^N$, and $\xi$ be the dual coordinates to $x$. $D_{\xi}$ is then the unique differential operator of the form 
$\sum_{j=1}^N c_j(x,\xi) \frac{\partial}{\partial \xi_j},$
such that
$$
D_{\xi} [a_0(M_x \xi)] = (\partial_{\xi_N} a_0) (M_x \xi) \quad \text{for all functions $a_0(\xi)$, and all $x$, $\xi$}.
$$
It follows that
\begin{equation} \label{eq:Dxidef}
D_{\xi} = \sum_{j=1}^N B_j^N(x) \frac{\partial}{\partial \xi_j}.
\end{equation}

Next, for $i=1,\dots,N$, $D_i$ is the unique differential operator of the form 
$\frac{\partial}{\partial x^i} + \sum_{j=1}^N b_{ij}(x,\xi) \frac{\partial}{\partial \xi_j}$
such that
$$
D_i [a_0(M_x \xi)] = 0 \quad \text{for all functions $a_0(\xi)$, and all $x$, $\xi$}.
$$
We have 
\begin{equation} \label{eq:Didef}
D_i = \frac{\partial}{\partial x^i} + \sum_{k=1}^N \sum_{p=1}^N \sum_{l=1}^N \frac{\partial B_k^p}{\partial x^i}(x) A_p^l(x) \xi_l \frac{\partial}{\partial \xi_k},
\end{equation}
and we write $D^I$ for $D_{i_1} D_{i_2} \dots D_{i_k}$, if $I = (i_1, \dots, i_k)$, where each $i_j \in \{1, \dots, N\}$.

\subsection{An example} 

If we identify $\mathbb{R}^N$ with the Heisenberg group $\mathbb{H}^n$, with $N = 2n+1$, and if $\mathcal{D}$ is the `horizontal' subspace of the tangent space spanned by the horizontal vector fields
$$
X_i = \frac{\partial}{\partial x^i} + 2 x^{i+n} \frac{\partial}{\partial t}, \quad X_{i+n} = \frac{\partial}{\partial x^{i+n}} - 2 x^i \frac{\partial}{\partial t}, \quad i = 1, \dots, n,
$$
(here $t = x^{2n+1}$ is the last coordinate on $\mathbb{R}^N$), then 
$$
\rho_x(\xi) \simeq \sum_{i=1}^n |\xi_i + 2x^{i+n} \tau| + |\xi_{i+n} - 2x^i \tau|
$$
(here $\tau$ is the last coordinate to $\xi$). Furthermore, if $x, y$ are on $\mathbb{R}^N$ are sufficiently close to each other, then writing $w$ for the product $y^{-1} x$ on the Heisenberg group, and $z$, $t$ for the first $2n$ and last coordinate of $w$ respectively, we get
$$
d(x,y) \simeq |z| + |t|^{1/2}.
$$
If we write $X_{2n+1} = \frac{\partial}{\partial t}$, the special derivatives $D_{\xi}$ and $D_i$ are given by
$$
D_{\xi} = \frac{\partial}{\partial \tau} + \sum_{j=1}^n \left(2 x^j  \frac{\partial}{\partial \xi_{j+n}} - 2 x^{j+n} \frac{\partial}{\partial \xi_j} \right),
$$
and
$$
D_i = \frac{\partial}{\partial x^i} + 2 \tau \frac{\partial}{\partial \xi_{i+n}}, \quad D_{i+n} = \frac{\partial}{\partial x^{i+n}} - 2 \tau \frac{\partial}{\partial \xi_i}, \quad i = 1, \dots, n.
$$

\section{The algebra} 

\subsection{Two ways of defining the symbol class}

We can now define the symbols of mixed homogeneity. We will fix a distribution $\mathcal{D}$ of tangent subspaces of $\mathbb{R}^N$, and fix a frame of tangent vectors $X_1, \dots, X_N$ on $\mathbb{R}^N$ that satisfies the conditions laid out in Section~\ref{sect:prelim}. Given any $m, n$, we say that $a(x,\xi) \in S^{m,n}$, if $a(x,\xi) \in C^{\infty}(\mathbb{R}^N \times \mathbb{R}^N)$, and satisfies
\begin{equation} \label{eq:symbolclassmn}
|\partial_{\xi}^{\alpha} D_{\xi}^{\beta} D^I a(x,\xi)| \lesssim_{\alpha,\beta,I} (1+|\xi|)^{m-|\beta|} (1+\rho_x(\xi)+|\xi|^{1/2})^{n-|\alpha|}.
\end{equation}

Equivalently, $a(x,\xi) \in S^{m,n}$, if and only if 
$$
a(x,\xi) = a_0(x,M_x \xi)
$$
for some symbol $a_0(x,\xi)$ that satisfies
\begin{equation} \label{eq:symbolclassmnequiv}
|\partial_{\xi'}^{\alpha} \partial_{\xi_N}^{\beta} \partial_x^I a_0(x,\xi)| \lesssim_{\alpha,\beta,I} (1+|\xi|)^{m-|\beta|} (1+\|\xi\|)^{n-|\alpha|},
\end{equation}
where $\|\xi\| := |\xi'|+|\xi_N|^{1/2}$ if $\xi = (\xi',\xi_N)$.

If we want to emphasize the dependence of our symbol class on $\mathcal{D}$, we write $S^{m,n}(\mathcal{D})$ instead of $S^{m,n}$. Note that if we write $\mathcal{D}^0$ for the `constant' distribution on $\mathbb{R}^N$ given by 
$$
\mathcal{D}^0_x = \text{kernel of $dx^N$ at every $x \in \mathbb{R}^N$},
$$ then (\ref{eq:symbolclassmnequiv}) is equivalent to saying that $a_0 \in S^{m,n}(\mathcal{D}^0)$.

We remark that the order of derivatives appearing in (\ref{eq:symbolclassmn}) does not matter, as can be seen from the equivalent characterization (\ref{eq:symbolclassmnequiv}). Note also that these symbols form a subclass of those standard symbols of type $(1/2,1/2)$. Furthermore, the above symbol class is defined independent of the choice of the vector fields $X_1, \dots, X_N$. It is also invariant under changes of coordinates: suppose $\tilde{x} = \Phi(x)$ is another coordinate system on $\mathbb{R}^N$, where the derivatives of $\frac{\partial \Phi}{\partial x}$ are uniformly bounded above, and the determinant of $(\frac{\partial \Phi}{\partial x})$ is uniformly bounded from below. We say in this case the change of coordinate is \textbf{admissible}. One can show that our symbol class is invariant under admissible changes of coordinates. In fact more is true: the class of operators we associate to $S^{m,n}(\mathcal{D})$ depends only on $\mathcal{D}$, and not on the choice of a coordinate system on $\mathbb{R}^N$. (See Theorem~\ref{thm:admisscoord} below.)

\subsection{Two special subclasses of the symbol class}

Within our class of symbols $S^{m,n}$, there are two particularly interesting subclasses, one of which is purely `isotropic', which we denote by $S^m$, another of which is purely `non-isotropic', which we denote by $S^n_{\mathcal{D}}$. We are thus led to think of our class $S^{m,n}$ as a class of symbols of mixed homogeneity.

First, given $m \in \mathbb{R}$, we take $S^m$ to be the class of all smooth functions $a(x,\xi)$ on the cotangent bundle of $\mathbb{R}^N$ such that
$$
|\partial_{\xi}^{\alpha} \partial_x^I a(x,\xi)| \lesssim_{\alpha,I} (1+|\xi|)^{m-|\alpha|}.
$$
This class is sometimes also called the \textbf{`isotropic' symbols} (or \textbf{classical symbols} of type $(1,0)$) of order $m$ . It is defined independent of the distribution $\mathcal{D}$. We also write $S^{-\infty} = \bigcap_{m \in \mathbb{N}} S^{-m}$.

Next, suppose we are given a distribution $\mathcal{D}$ of tangent subspaces on $\mathbb{R}^N$, with a frame of tangent vectors $X_1, \dots, X_N$ as above. For any $n \in \mathbb{R}$, we then take $S^n_{\mathcal{D}}$ to be the class of \textbf{`non-isotropic' symbols} of order $n$, which is the class of all smooth functions $a(x,\xi)$ on the cotangent bundle of $\mathbb{R}^N$ such that
$$
|\partial_{\xi}^{\alpha} D_{\xi}^{\beta} D^I a(x,\xi)| \lesssim_{\alpha,\beta,I} (1+\rho_x(\xi)+|\xi|^{1/2})^{n-|\alpha|-2|\beta|}.
$$ 
Equivalently, $a(x,\xi) \in S^n_{\mathcal{D}}$, if and only if $a(x,\xi) = a_0(x,M_x \xi)$ for some $a_0(x,\xi) \in S^n$ that satisfies
$$
|\partial_{\xi'}^{\alpha} \partial_{\xi_N}^{\beta} \partial_x^I a_0(x,\xi)| \lesssim_{\alpha,\beta,I} (1+\|\xi\|)^{n-|\alpha|-2|\beta|}.
$$
Note that the latter is the same as saying $a_0 \in S^n_{\mathcal{D}^0}$, and the order of derivatives appearing in the definitions does not matter as with the case of $S^{m,n}$.

One observes easily that
$$S^m \subset S^{m,0}, \quad S^n_{\mathcal{D}} \subset S^{0,n},$$
and that
$$
S^{m,n} \subset S^{m,n'} \quad \text{if $n < n'$}, \quad S^{m,n} \subset S^{m',n} \quad \text{if $m < m'$}.
$$
Furthermore, one has, for any $m,n \in \mathbb{R}$,
\begin{equation} \label{eq:symbol_inclusion}
S^{m+\varepsilon, n-2\varepsilon} \subset S^{m,n} \quad \text{and} \quad S^{m-\varepsilon, n+\varepsilon} \subset S^{m,n} \quad \text{for all $\varepsilon > 0$}.
\end{equation}

\subsection{Main theorems}

Now suppose $\mathcal{D}$ is a distribution of tangent subspaces on $\mathbb{R}^N$, and that we are given a symbol $a(x,\xi) \in S^{m,n}(\mathcal{D})$. We define the associated pseudodifferential operator 
$$T_af(x) = \int_{\mathbb{R}^N} a(x,\xi) \widehat{f}(\xi) e^{2\pi i x \cdot \xi} d\xi$$
where $f \in \s$. Our first theorem shows that our class of pseudodifferential operators form an algebra:

\begin{thm} \label{thm:compose}
If $a_1 \in S^{m_1,n_1}$ and $a_2 \in S^{m_2,n_2}$, then $$T_{a_1} \circ T_{a_2} = T_a$$ for some symbol $a \in S^{m,n}$, where $m = m_1 + m_2$, $n = n_1 + n_2$.
\end{thm}

There is no explicit asymptotic development for $a$ in terms of $a_1$ and $a_2$. However, when either $a_1$ or $a_2$ is an `isotropic' symbol, then because the symbols of the class $S^{m,n}$ are of type $(1/2,1/2)$, the Kohn-Nirenberg formula continues to hold.

Our second theorem shows that our class of pseudodifferential operators is closed under adjoints:

\begin{thm} \label{thm:adjoint}
Let $a \in S^{m,n}$, and $T_a^*$ be the adjoint of $T_a$ with respect to the standard $L^2$ inner product on $\mathbb{R}^N$. Then there exists a symbol $a^* \in S^{m,n}$ such that 
$$T_a^* = T_{a^*}.$$
\end{thm}

The next result is about representation of pseudodifferential operators when some $\Theta$ compatible with our distribution $\mathcal{D}$ is given. From the characterization of $S^{m,n}(\mathcal{D})$ in (\ref{eq:symbolclassmnequiv}), it is clear that if $a \in S^{m,n}(\mathcal{D})$, then there exists $a_0 \in S^{m,n}(\mathcal{D}^0)$ such that when interpreted suitably,
\begin{equation} \label{eq:Theta0rep}
T_a f(x) = \int_{\mathbb{R}^N} \int_{\mathbb{R}^N} a_0(x,\xi) f(y) e^{2\pi i \Theta_0(x,y) \cdot \xi} dy d\xi
\end{equation}
for all $f \in \s$ and all $x \in \mathbb{R}^N$. It turns out that this holds for all $\Theta$ that are compatible with $\mathcal{D}$, at least locally. To describe this, we recall the following notion: we say $E$ is an \textbf{infinitely smoothing operator}, if $E = T_e$ for some symbol $e \in S^{-\infty}$.

\begin{thm} \label{thm:Thetarep}
Suppose $\Theta \colon \Delta \to \mathbb{R}^N$ is a map compatible with $\mathcal{D}$. Then there exists an absolute constant $\delta_0 > 0$ such that for any $\phi \in C^{\infty}_c(B(0,\delta_0))$ with $\phi = 1$ on $B(0,\delta_0/2)$, the following holds:
\begin{enumerate}[(a)]
\item For any $a \in S^{m,n}(\mathcal{D})$, there exists $a_0 \in S^{m,n}(\mathcal{D}^0)$, such that
\begin{equation} \label{eq:a0repTheta}
T_a f(x) = \int_{\mathbb{R}^N} \int_{\mathbb{R}^N} \phi(x-y) a_0(x,\xi) f(y) e^{2\pi i \Theta(x,y) \cdot \xi} dy d\xi + Ef(x),
\end{equation}
where $E$ is an infinitely smoothing operator. 
\item Conversely, for any $a_0 \in S^{m,n}(\mathcal{D}^0)$, and any infinitely smoothing operator $E$, there exists $a \in S^{m,n}(\mathcal{D})$, such that (\ref{eq:a0repTheta}) holds.
\end{enumerate}
\end{thm}

Finally, we have the following theorem, which shows that the class of operators associated to symbols of order $(m,n)$ is invariant under admissible changes of coordinates.

\begin{thm} \label{thm:admisscoord}
Suppose $\tilde{x} = \Phi(x)$ is an admissible coordinate system on $\mathbb{R}^N$. If $a \in S^{m,n}(\mathcal{D})$, then there exists $\tilde{a} \in S^{m,n}(d\Phi(\mathcal{D}))$, such that
$$
T_a f(x) = T_{\tilde{a}} \tilde{f} (\Phi(x))
$$
for all $f \in \s$, where $\tilde{f} := f \circ \Phi^{-1}$.
\end{thm}

It is convenient to introduce compound symbols in the proof of these theorems, to which we now turn.

\subsection{Compound symbols} \label{subsect:compound}

To define compound symbols, we proceed in two steps. First we consider a class of preliminary compound symbols, which we denote by $PCS^{m,n}$. Then we pass to the full class of compound symbols $CS^{m,n}$. Here $(m,n)$ denote the orders of the symbols.

To begin with, we extend the $D_{\xi}$ and $D_i$ defined earlier, so that they become differential operators on functions of $(x,y,\xi) \in \mathbb{R}^N \times \mathbb{R}^N \times \mathbb{R}^N$. To do so, let
$$
D_i = \frac{\partial}{\partial y^i} + \frac{\partial}{\partial x^i} + \sum_{k=1}^N \sum_{p=1}^N \sum_{l=1}^N \frac{\partial B_k^p}{\partial x^i}(x) A_p^l(x) \xi_l \frac{\partial}{\partial \xi_k}.
$$
This is consistent with the old notion of $D_i$ as in (\ref{eq:Didef}), since if this new $D_i$ acts on a function that is independent of $y$, then its action is just the same as that of the old $D_i$. We will still write
$$
D_{\xi} = \sum_{j=1}^N B_j^N(x) \frac{\partial}{\partial \xi_j}
$$ 
as in (\ref{eq:Dxidef}), and let it differentiate a function of $x,y,\xi$  without differentiating the $y$ variable. 

We will say that a function $\rho_{x,y}(\xi)$, defined for $x,y \in \mathbb{R}^N$ and all $\xi \in \mathbb{R}^N$, is \textbf{compatible} with $\rho_x(\xi)$, if
\begin{equation} \label{eq:rhoxycond}
|\rho_{x,y}(\xi)-\rho_x(\xi)| \leq C |x-y||\xi| \quad \text{for all $x,y, \xi$.}
\end{equation} 
For example, $\rho_{x,y}(\xi) = \rho_x(\xi)$ will do, and so will $\rho_{x,y}(\xi) = \rho_y(\xi)$.

Now given two real numbers $m$ and $n$, 
%if $c(x,y,\xi)$ is a function in $C^{\infty}(\mathbb{R}^N \times \mathbb{R}^N \times \mathbb{R}^N)$ that satisfies
%\begin{equation} \label{eq:restcompsymb}
%|\partial_{\xi}^{\alpha} \partial_y^{\gamma} \partial_x^{\delta} c(x,y,\xi)| \lesssim_{\alpha,\gamma,\delta} (1+|\xi|)^{m+\frac{|\gamma|+|\delta|-|\alpha|}{2}} (1+\rho_{x,y}(\xi)+|\xi|^{1/2})^{n}.
%\end{equation}
%for some $\rho_{x,y}(\xi)$ that is compatible with $\rho_x(\xi)$, we say that it is a \textbf{elementary compound symbol} of order $(m,n)$. More generally, 
suppose $c(x,y,\xi)$ is a finite sum of functions $c_i(x,y,\xi)$, and suppose for each $i$, there exist $k$ real numbers $n_1, \dots, n_k$ with $n_1 + \dots + n_k = n$, and $k$ functions $\rho^{(1)}_{x,y}(\xi), \dots, \rho^{(k)}_{x,y}(\xi)$, each of which is compatible with $\rho_x(\xi)$, such that
\begin{equation} \label{eq:restcompsymb2}
|\partial_{\xi}^{\alpha} \partial_y^{\gamma} \partial_x^{\delta} c_i(x,y,\xi)| 
\lesssim_{\alpha,\gamma,\delta} (1+|\xi|)^{m+\frac{|\gamma|+|\delta|-|\alpha|}{2}} \prod_{j=1}^k (1+\rho^{(j)}_{x,y}(\xi)+|\xi|^{1/2})^{n_j}.
\end{equation}
Then we say $c(x,y,\xi)$ is a \textbf{preliminary compound symbol} of order $(m,n)$, and we write $c \in PCS^{m,n}$.

A \textbf{compound symbol} of order $(m,n)$ is then a function $c(x,y,\xi) \in C^{\infty}(\mathbb{R}^N \times \mathbb{R}^N \times \mathbb{R}^N)$ such that for each multiindices $\beta$ and $I$, we have a decomposition of $\partial_{\xi}^{\alpha} D_{\xi}^{\beta} D^I c(x,y,\xi)$ into a finite sum
\begin{equation} \label{eq:compoundexpansion}
\partial_{\xi}^{\alpha} D_{\xi}^{\beta} D^I c(x,y,\xi) = \sum_{\sigma} (x-y)^{\sigma} c_{\sigma}^{\alpha,\beta,I}(x,y,\xi)
\end{equation}
where each $$c_{\sigma}^{\alpha,\beta,I} \in PCS^{m-|\beta|+\frac{|\sigma|}{2},n-|\alpha|}.$$ 
%(The choice of $\rho_{x,y}(\xi)$ involved when we say $c_{\sigma}^{\alpha,\beta,I} \in PCS^{m-\beta+\frac{|\sigma|}{2},n-|\alpha|}$ may be different for different $\alpha, \beta, I, \sigma$.) 
We denote by $CS^{m,n}$ the class of compound symbols of order $(m,n)$. 

%Also, if $c(x,y,\xi)$ is a compound symbol of order $(m,n)$, then by differentiating (\ref{eq:compoundexpansion}), automatically we have, for each multiindices $\alpha, \beta, \gamma, \delta, I$, a decomposition of $\partial_{\xi}^{\alpha} D_{\xi}^{\beta} \partial_y^{\gamma} \partial_x^{\delta} D^I c$ into a finite sum
%\begin{equation} \label{eq:compoundexpansion2}
%\partial_{\xi}^{\alpha} D_{\xi}^{\beta} \partial_y^{\gamma} \partial_x^{\delta} D^I c(x,y,\xi) = \sum_{\sigma} (x-y)^{\sigma} c_{\sigma}^{\alpha, \beta, \gamma, \delta, I}(x,y,\xi)
%\end{equation}
%where each $$c_{\sigma}^{\alpha, \beta, \gamma, \delta, I} \in PCS^{m-\beta+\frac{|\gamma|+|\delta|+|\sigma|}{2},n-|\alpha|}.$$ Furthermore, the order of derivatives in (\ref{eq:compoundexpansion2}) does not matter.

%Finally, the class of compound symbols is actually rather symmetric with respect to the $x,y$ variables: in fact, if $c(x,y,\xi)$ is a preliminary compound symbol with `norm' function $\rho_{x,y}(\xi)$, then $c(y,x,\xi)$ is a preliminary compound symbol with `norm' function $\rho_{y,x}(\xi)$. Similarly for the full class of compound symbols.

Given $c(x,y,\xi) \in CS^{m,n}$, we define
\begin{align*}
T_c f(x) &= \int_{\mathbb{R}^N} \int_{\mathbb{R}^N} c(x,y,\xi) f(y) e^{2\pi i (x-y) \cdot \xi} d\xi dy 
%\\ &= \int_{\mathbb{R}^N} \int_{\mathbb{R}^N} \frac{1}{(2\pi i)^{|\alpha|}} \sum_{\sigma} \partial_{\xi}^{\sigma} c_{\sigma} (x,y,\xi) f(y) e^{2\pi i (x-y) \cdot \xi} d\xi dy
\end{align*} 
for all $f \in \s$. This makes sense if $c(x,y,\xi)$ is compactly supported in $y$ and $\xi$; if it is not, then we approximate $c(x,y,\xi)$ by symbols that are. Our main result about $T_c$ is the following theorem.

\begin{thm} \label{thm:compoundtosymbol}
If $c(x,y,\xi) \in CS^{m,n}$, then there exists a symbol $a(x,\xi) \in S^{m,n}$ such that $$T_c f = T_a f$$ for all $f \in \s$.
\end{thm}

We may now sketch the proofs of Theorems~\ref{thm:compose}, \ref{thm:adjoint},  \ref{thm:Thetarep} and \ref{thm:admisscoord}.

\begin{proof}[Proof of Theorem~\ref{thm:adjoint}]
Given $a \in S^{m,n}$, let $$c(x,y,\xi) = \overline{a(y,\xi)}.$$ Then $c(x,y,\xi)$ is a compound symbol in $CS^{m,n}$; the relevant $\rho_{x,y}(\xi)$ in this case is given by $$\rho_{x,y}(\xi) := \rho_y(\xi).$$ Now $T_a^* f = T_c f$ for all $f \in \s$. Thus Theorem~\ref{thm:adjoint} follows from Theorem~\ref{thm:compoundtosymbol}.
\end{proof}

\begin{proof}[Proof of Theorem~\ref{thm:compose}]
Given $a_i \in S^{m_i,n_i}$, $i = 1,2$, write $T_{a_2} = T_{a_2^*}^*$ for some symbol $a_2^* \in S^{m_2,n_2}$ by Theorem~\ref{thm:adjoint}. Let $$c(x,y,\xi) = a_1(x,\xi) a_2^*(y,\xi).$$ Then $c(x,y,\xi)$ is a compound symbol in $CS^{m,n}$, with $m = m_1 + m_2$, $n = n_1 + n_2$; the relevant $\rho^{(j)}_{x,y}(\xi)$ in this case are given by
$$
\rho^{(1)}_{x,y}(\xi) := \rho_x(\xi), \quad \rho^{(2)}_{x,y}(\xi) := \rho_y(\xi).
$$
Now
$$T_{a_1} T_{a_2} = T_{a_1} T_{a_2^*}^* = T_c,$$
so Theorem~\ref{thm:compose} follows from Theorem~\ref{thm:compoundtosymbol}.
\end{proof}

\begin{proof}[Proof of Theorem~\ref{thm:Thetarep}]
Given a compatible $\Theta$, write
\begin{equation} \label{eq:calL}
\Theta(x,y) = \mathcal{L}_{x,x-y}(x-y)
\end{equation}
for all $(x, y) \in \Delta$, where $\mathcal{L}_{x,u}$ is a variable coefficient linear map defined for each $x \in \mathbb{R}^N$, $|u| \leq \delta$. To prove (a), we write
$$u = \mathcal{M}_x(v) \quad \text{whenever $v = \mathcal{L}_{x,u}(u)$ and $|u|$ is sufficiently small.}$$ Furthermore, given $a \in S^{m,n}(\mathcal{D})$, we claim the existence of $a_0 \in S^{m,n}(\mathcal{D}^0)$ such that 
$$
\det(\mathcal{L}_{x,\mathcal{M}_x(v)}) \int a(x,\mathcal{L}^t_{x,\mathcal{M}_x(v)} \xi) e^{2 \pi i v \cdot \xi} d\xi = \int a_0(x,\xi) e^{2\pi i v \cdot \xi} d\xi
$$
for all $x \in \mathbb{R}^N$ and $|v|$ sufficiently small. If this is true, then setting $v = \mathcal{L}_{x,x-y}(x-y)$, we have, by (\ref{eq:calL}), that
$$
\int a_0(x,\xi) e^{2\pi i \Theta(x,y) \cdot \xi} d\xi
=\int a(x,\xi) e^{2 \pi i (x-y) \cdot \xi} d\xi 
$$
for all $x, y \in \mathbb{R}^N$ with $|x-y|$ sufficiently small. The rest then follows easily.

The claim again follows from Theorem~\ref{thm:compoundtosymbol}. Let $\tilde{\phi} \in C^{\infty}_c(\mathbb{R}^N)$ be identically 1 near 0. Then 
$$
c(x,y,\xi):= \tilde{\phi}(x-y) \det(\mathcal{L}_{x,\mathcal{M}_x(x-y)}) a(x,\mathcal{L}^t_{x,\mathcal{M}_x(x-y)} \xi)
$$
is a compound symbol of order $(m,n)$ adapted to the distribution $\mathcal{D}^0$, with 
$$
\rho_{x,y}(\xi) := \rho_x \left(\mathcal{L}^t_{x,\mathcal{M}_x(x-y)} \xi \right) = \left|\left(M_x \mathcal{L}^t_{x,\mathcal{M}_x(x-y)} \xi \right)'\right|.
$$
Thus there exists $a_0 \in S^{m,n}(\mathcal{D}^0)$ such that
$$
\int \int c(x,y,\xi) e^{2\pi i(x-y) \cdot \xi} d\xi = \int a_0(x,\xi) e^{2 \pi i (x-y) \cdot \xi} d\xi
$$
for all $x,y \in \mathbb{R}^N$, which implies our desired claim if we replace $x-y$ by $v$.

Conversely, given $a_0 \in S^{m,n}(\mathcal{D}^0)$, let
$$
c(x,y,\xi) := \phi(x-y) \det(\mathcal{L}_{x,x-y}^{-1}) a_0(x,(\mathcal{L}_{x,x-y}^{-1})^t \xi)
$$
where $\phi$ is a cut-off so that $\mathcal{L}_{x,x-y}$ is invertible when $\phi(x-y) \ne 0$.
Then $c(x,y,\xi) \in CS^{m,n}(\mathcal{D})$ with $$\rho_{x,y}(\xi) := |\left( (\mathcal{L}_{x,x-y}^{-1})^t \xi \right)'|.$$ Hence by Theorem~\ref{thm:compoundtosymbol}, there exists $a \in S^{m,n}(\mathcal{D})$ such that $T_a = T_c$. The rest then follows easily.
\end{proof}

\begin{proof}[Proof of Theorem~\ref{thm:admisscoord}]
Let $\phi \in C^{\infty}_c(\mathbb{R}^N)$ be a cut-off that is identically 1 near $0$. Given $a \in S^{m,n}(\mathcal{D})$, write 
$$
T_1 f(x) = \int \int \phi(x-y) a(x,\xi) f(y) e^{2\pi i (x-y) \cdot \xi} dy d\xi.
$$
Then $T_a = T_1$ modulo an infinitely smoothing operator, and the key is to show that there exists $\tilde{a}_1 \in S^{m,n}(d\Phi(\mathcal{D}))$, such that $T_1 f(x) = T_{\tilde{a}_1} \tilde{f} (\Phi(x))$. Now write $\Psi = \Phi^{-1}$, and write
$$
\Psi(\tilde{x})-\Psi(\tilde{y}) = \mathcal{L}_{\tilde{x},\tilde{x}-\tilde{y}}(\tilde{x}-\tilde{y})
$$
for some variable linear map $\mathcal{L}_{\tilde{x},\tilde{u}}$. By restricting the support of $\phi$, we may assume that $\mathcal{L}_{\tilde{x},\tilde{x}-\tilde{y}}$ is invertible whenever $\phi(\Psi(\tilde{x})-\Psi(\tilde{y})) \ne 0$. Thus if we let $\tilde{x} = \Phi(x)$, then for any $f \in \s$,
\begin{align*}
T_1 f(x) 
= \int \int \tilde{c}(\tilde{x},y,\xi) \tilde{f}(y) e^{2\pi i (\tilde{x}-y) \cdot \xi} dy d\xi
\end{align*}
where
$$
\tilde{c}(x,y,\xi) := \phi(\Psi(x)-\Psi(y)) a(\Psi(x),(\mathcal{L}_{x,x-y}^{-1})^t \xi) \det(\Psi'(y)) \det(\mathcal{L}_{x,x-y}^{-1}).
$$
\newpage \noindent
But $\tilde{c}(x,y,\xi) \in CS^{m,n}(d\Phi (\mathcal{D}))$, with 
$$
\rho_{x,y}(\xi) = \rho_{\Psi(x)} \left( (\mathcal{L}_{x,x-y}^{-1})^t \xi \right).
$$
Thus the existence of the desired $\tilde{a}_1$ follows from  Theorem~\ref{thm:compoundtosymbol}.
\end{proof}

\begin{proof}[Proof of Theorem~\ref{thm:compoundtosymbol}]
We will assume for simplicity that $c(x,y,\xi)$ has compact support in $y$ and $\xi$, and so do all the $c_{\sigma}^{\alpha,\beta,I}(x,y,\xi)$ arising in the expansion (\ref{eq:compoundexpansion}). Then $T_c f = T_a f$ for all $f \in \s$, where 
\begin{equation} \label{eq:compoundtosymbol}
a(x,\xi) = \int_{\mathbb{R}^N} \int_{\mathbb{R}^N} c(x,x-w,\xi-\zeta) e^{-2\pi i w \cdot \zeta} dw d\zeta.
\end{equation}
As a shorthand we write $$a = [c]$$ when this identity holds.
Then we write $$c(x,y,\xi) = \sum_{\sigma} (x-y)^{\sigma} c_{\sigma}(x,y,\xi) $$ and substitute into (\ref{eq:compoundtosymbol}).
Writing $-2 \pi i w e^{-2\pi i w \cdot \zeta} = \partial_{\zeta} e^{-2\pi i w \cdot \zeta}$ and integrating by parts, we get
$$a(x,\xi) = \sum_{\sigma} (2\pi i)^{-|\sigma|} [\partial_{\xi}^{\sigma} c_{\sigma}].$$
But now each $$\partial_{\xi}^{\sigma} c_{\sigma} (x,y,\xi) \in  PCS^{m,n}.$$ Thus to show the $L^{\infty}$ bound of $a(x,\xi)$, one only needs to invoke Proposition~\ref{prop:restrictedcompoundtosymbol} below. Similarly one can estimate all derivatives of $a(x,\xi)$, namely $\partial_{\xi}^{\alpha} D_{\xi}^{\beta} D^I a(x,\xi)$, in $L^{\infty}$. We omit the details.
\end{proof}

\begin{prop} \label{prop:restrictedcompoundtosymbol}
Suppose $c(x,y,\xi) \in PCS^{m,n}$ has compact support in $y$ and $\xi$. Then $a(x,\xi)$, defined by (\ref{eq:compoundtosymbol}), satisfies $$|a(x,\xi)| \lesssim (1+|\xi|)^m (1+\rho_x(\xi)+|\xi|^{1/2})^n.$$
\end{prop}

\begin{proof}
Heuristically, one performs a Taylor expansion in $\zeta$ in (\ref{eq:compoundtosymbol}). More precisely, suppose $|\xi| \geq 1$. Then we choose a smooth function $\phi$ on $\mathbb{R}$ that is supported on $(-1,1)$, and that is identically equal to 1 on $(-1/2,1/2)$, so that $a(x,\xi)$ is equal to
\begin{equation} \label{eq:PCSphi}
\int \int \phi \left( \frac{|\zeta|}{|\xi|/2} \right) c(x, x-w, \xi-\zeta) e^{-2\pi i w \cdot \zeta} dw d\zeta
\end{equation}
up to an error that is rapidly decreasing in $\xi$. To estimate (\ref{eq:PCSphi}), we write the exponential in the integral as
$$
e^{-2\pi i w \cdot \zeta} = \left( \frac{I - |\xi|^{-1} \Delta_w - |\xi| \Delta_{\zeta}}{1 + 4\pi^2 |\xi| |w|^2 + 4\pi^2 |\xi|^{-1} |\zeta|^2} \right)^{M} e^{-2\pi i w \cdot \zeta},
$$
and integrate by parts. On the support of this integral, $|\zeta| \leq |\xi|/2$, so in particular $|\xi-\zeta| \simeq |\xi|$. Using this, and estimates on the derivatives of $c$, we see that (\ref{eq:PCSphi}) is bounded by 
\begin{align*}
\int_{|\zeta| \leq |\xi|/2} \int  \frac{(1+|\xi|)^{m} (1+\rho_{x,x-w}(\xi-\zeta)+|\xi|^{1/2})^{n}}{(1+4\pi^2 |\xi| |w|^2 + 4\pi^2 |\xi|^{-1} |\zeta|^2)^{M}} dw d\zeta
\end{align*}
if $c$ satisfies (\ref{eq:restcompsymb2}) with $k=1$.
Now if $n \geq 0$, on the support of $II$, we bound
\begin{equation} 
1+\rho_{x,x-w}(\xi-\zeta)+|\xi|^{1/2}
\lesssim (1+\rho_x(\xi) + |\xi|^{1/2}) (1 + |\xi|^{-1/2} |\zeta| +  |\xi|^{1/2} |w|).
\end{equation}
%\begin{align*}
%1+\rho_{x,x-w}(\xi-\zeta)+|\xi|^{1/2}
%& \lesssim 1+\rho_x(\xi-\zeta) + |w| |\xi-\zeta| + |\xi|^{1/2} \\
%& \lesssim 1+\rho_x(\xi) + |\zeta| + |\xi| |w| + |\xi|^{1/2} \\
%& \lesssim 1+\rho_x(\xi) + |\xi|^{1/2}(1 + |\xi|^{-1/2} |\zeta| +  |\xi|^{1/2} |w|) \\
%& \lesssim (1+\rho_x(\xi) + |\xi|^{1/2}) (1 + |\xi|^{-1/2} |\zeta| +  |\xi|^{1/2} |w|).
%\end{align*}
If $n < 0$, on the support of $II$, we use the bound 
\begin{equation} 
\frac{1}{1+\rho_{x,x-w} (\xi-\zeta)+|\xi|^{1/2}} \lesssim \frac{(1+|\xi|^{-1/2}|\zeta|) (1+|\xi|^{1/2} |w|)}{1+\rho_x(\xi)+|\xi|^{1/2}}.
\end{equation}
Altogether, (\ref{eq:PCSphi}) is bounded by $(1+|\xi|)^m (1+\rho_x(\xi)+|\xi|^{1/2})^n.$
One can easily adapt this argument if $c$ satisfies (\ref{eq:restcompsymb2}) for a general $k$, or if $c$ is a sum of terms each satisfying (\ref{eq:restcompsymb2}). Thus we are done in this case. The case when $|\xi| \leq 1$ is much easier.
\end{proof}

\section{Kernel representations}

In this section, we describe differential inequalities for the kernels of pseudodifferential operators arising from symbols of class $S^{m,n}$. We also describe a partial converse, which states the extent to which these differential inequalities on the kernels characterize the pseudodifferential operators. We assume throughout this section that 
$$m > -1, \quad n > -(N-1).$$

\subsection{The good derivatives on $\mathbb{R}^N$}

In order to formulate our kernel estimates, we need to identify some good derivatives on $\mathbb{R}^N$, which are in some sense dual to the ones we have in Section~\ref{subsect:goodcotder} and~\ref{subsect:compound}. 

First, the $X_1, \dots, X_{N-1}$, which we introduced in Section~\ref{subsect:assump}, are the good $x$-derivatives on $\mathbb{R}^N$. When we want to distinguish derivatives in the $x$ and the $y$ variables, we put a subscript of $x$ or $y$ respectively. e.g.
$$
(X_i)_y := \sum_{j=1}^N A_i^j(y) \frac{\partial}{\partial y^j}.
$$
We write $X'$ for any of the $X_1, \dots, X_{N-1}$'s.

We need $N$ further derivatives in both $x$ and $y$, which we denote by $D_{x,y,i}$, $i=1,\dots,N$. Here $D_{x,y,i}$ is the unique differential operator of the form 
$\frac{\partial}{\partial x^i} + \sum_{j=1}^N b_{ij}(x,y) \frac{\partial}{\partial y^j},$
such that
$$
D_{x,y,i} [F(\Theta_0(x,y))] = 0 \quad \text{for all functions $F$, and all $x$, $y$}.
$$
It follows that
$$
D_{x,y,i} = \frac{\partial}{\partial x^i} + \frac{\partial}{\partial y^i} + \sum_{k=1}^N \sum_{p=1}^N \sum_{l=1}^N \frac{\partial B_k^p}{\partial x^i}(x) A_p^l(x) (x^k-y^k) \frac{\partial}{\partial y^l}.
$$
If $I = (i_1, \dots, i_k)$ with each $i_j \in \{1, \dots, N\}$, we also write $D_{x,y}^I = D_{x,y,i_1} \dots D_{x,y,i_k}$.

\subsection{Main theorems}

\begin{thm} \label{thm:KernelTheta}
Suppose $m > -1$, $n > -(N-1)$. Let $a \in S^{m,n}$. Then there exists a distribution $K(x,y)$ on $\mathbb{R}^N \times \mathbb{R}^N$, such that the following holds:
\newpage \noindent
\begin{enumerate}[(a)]
\item \label{7a} We have $$T_a f(x) = \int_{\mathbb{R}^N} K(x,y) f(y) dy$$ for all $f \in \s$, in the sense that
$$\langle T_a f, g \rangle = \left \langle K(x,y) , f(y) g(x) \right \rangle$$
for all $f, g \in \s$.
\item \label{7b} The distribution $K(x,y)$ is smooth away from the diagonal $\{x=y\}$, and satisfies the differential inequalities
\begin{equation} \label{eq:kernelest1}
|D_{x,y}^I (X'_{x,y})^{\gamma} \partial_{x,y}^{\lambda} K(x,y)| \lesssim \frac{1}{|x-y|^{(N-1)+n+|\gamma|+M} \,  d(x,y)^{2(1+m+|\lambda|)}}
\end{equation}
for all $M \geq 0$, where by $X'_{x,y}$ we mean the derivative can be either with respect to $x$ or $y$; similarly for $\partial_{x,y}$.
\end{enumerate}
\end{thm}

Note that the `isotropic' index $m$, and the `non-isotropic' index $n$, of the symbols in $S^{m,n}$, now appear in reverse roles of in the estimates of the kernels.

There is a partial converse of the above theorem:

\begin{thm} \label{thm:KernelThetaEasyConverse}
Assume in addition $$m < 0, \quad n < 0.$$  Suppose $K(x,y)$ is a distribution on $\mathbb{R}^N \times \mathbb{R}^N$, and it satisfies part (\ref{7b}) of Theorem~\ref{thm:KernelTheta}. Then there exists $a \in S^{m,n}$, such that part (\ref{7a}) of Theorem~\ref{thm:KernelTheta} holds.
\end{thm}
 
On the other hand, when one of $m$ or $n$ is non-negative, the kernel $K(x,y)$ satisfies some cancellation conditions, which together with the kernel estimates in part (\ref{7b}) of Theorem~\ref{thm:KernelTheta}, characterize the kernels arising from symbols of the class $S^{m,n}$. We will only state the cancellation conditions in the case where $m = n = 0$.

A function $\phi$ is said to be a \textbf{normalized bump function}, if $\phi$ is smooth and supported on the unit ball, with $\|\partial^I \phi\|_{L^{\infty}} \leq C_I$ for all $I$.

\begin{thm} \label{thm:KernelS00}
\begin{enumerate}[(a)]
\item \label{9a} Suppose $a \in S^{0,0}$, and $K(x,y)$ is a distribution on $\mathbb{R}^N \times \mathbb{R}^N$ for which part (\ref{7a}) of Theorem~\ref{thm:KernelTheta} holds. Then there exists a distribution $k_0(x,u)$ on $\mathbb{R}^N \times \mathbb{R}^N$, with $$K(x,y) = k_0(x,\Theta_0(x,y)),$$ such that the following cancellation conditions hold:
whenever $\phi_1(u')$, $\phi_2(u'')$ are normalized bump functions, and $R_1, R_2 \geq 1$, $M' \geq 0$, we have
\begin{align} 
\label{eq:canccondz}
\left| \int_{\mathbb{R}^{N-1}} \partial_x^I \partial_{u''}^{\lambda} k_0(x,u) \phi_1(R_1 u') du' \right| & \lesssim \frac{1}{|u''|^{1+|\lambda|+M'}} \\
\label{eq:canccondt}
\left| \int_{\mathbb{R}} \partial_x^I \partial_{u'}^{\gamma} k_0(x,u) \phi_2(R_2 u'') du'' \right| & \lesssim \frac{1}{|u'|^{(N-1)+|\gamma|+M'}} \\
\label{eq:canccondzt}
\left| \int_{\mathbb{R}^N} \partial_x^I k_0(x,u) \phi_1(R_1 u') \phi_2(R_2 u'') du \right| & \lesssim 1.
\end{align} 
Here $u = (u',u'') \in \mathbb{R}^{N-1} \times \mathbb{R}$, and $R_1 u'$ and $R_2 u''$ are the Euclidean dilations of $u'$ and $u''$ respectively.
\item \label{9b} Conversely, if $K(x,y)$ is a distribution on $\mathbb{R}^N \times \mathbb{R}^N$ that satisfies part (\ref{7b}) of Theorem~\ref{thm:KernelTheta} with $m = n = 0$, and $K(x,y) = k_0(x,\Theta_0(x,y))$ for some distribution $k_0(x,u)$  on $\mathbb{R}^N \times \mathbb{R}^N$ that satisfies the cancellation conditions (\ref{eq:canccondz}), (\ref{eq:canccondt}) and (\ref{eq:canccondzt}), then there exists $a \in S^{m,n}$, such that part (\ref{7a}) of Theorem~\ref{thm:KernelTheta} holds.
\end{enumerate}
\end{thm}

We remark that the above three theorems are invariant under changes of frames, and under admissible changes of coordinates. 

The proofs of these theorems follow the philosophy of \cite{NRSW2012}, where the result of Theorem~\ref{thm:KernelS00} is proved in the context of 2-step nilpotent groups when $m = n = 0$.  The crux of the matter is the case where the distribution at hand is $\mathcal{D}^0$, which we focus from now on. In order to prove Theorem~\ref{thm:KernelTheta} and Theorem~\ref{thm:KernelS00}(\ref{9a}), we write $u = x-y$, and split $u = (u', u'') \in \mathbb{R}^{N-1} \times \mathbb{R}$. We break up the $u$ space near the origin into 3 regions, where $|u'| < |u''|$, where $|u'| > |u''|^{1/2}$, and the intermediate case where $|u''| \leq |u'| \leq |u''|^{1/2}$. The intermediate case then gives rise to the most critical part of the integral. In proving Theorem~\ref{thm:KernelThetaEasyConverse} and Theorem~\ref{thm:KernelS00}(\ref{9b}), we use a corresponding splitting of the $\xi$ space. The details are omitted.

\section{The $L^p$ theory}

Our main theorem in this section is the $L^p$ boundedness of pseudodifferential operators whose symbols are in the class $S^{0,0}$:

\begin{thm} \label{thm:Lptwoflag}
If $a(x,\xi) \in S^{0,0}$, then $T_a$ maps $L^p(\mathbb{R}^N)$ to $L^p(\mathbb{R}^N)$, for all $1 < p < \infty$.
\end{thm}

The proof resembles that of the corresponding result for flag kernels on the Heisenberg group $\mathbb{H}^n$; see e.g. \cite{MR1376298}  or \cite{MR2949616}. There one needs to use two versions of Littlewood-Paley decompositions, one of which is adapted to the non-isotropic (aka automorphic) dilations on $\mathbb{H}^n$, and another is a Littlewood-Paley projection in the central variable $t$. In what follows, we will also need two Littlewood-Paley decompositions. But we will use, instead of Littlewood-Paley projections in the single variable $t$, Littlewood-Paley projections that are `isotropic' in nature.

\subsection{The Littlewood-Paley decompositions}

We now turn to the two versions of Littlewood-Paley projections that we need.

The first version is an `isotropic' one, which we denote by $P_j$. Let $\phi \in C^{\infty}_c(\mathbb{R}^N)$ be such that 
$$ \phi(\xi)  = \begin{cases} 1 \quad \text{if $|\xi| \leq 1$} \\   0 \quad \text{if $|\xi| \geq 2$} \end{cases},$$ 
and define $$p_0(\xi) = \phi(\xi), \quad p_j(\xi) = \phi(2^{-j} \xi) - \phi(2^{-(j-1)} \xi) \quad \text{for $j \geq 1$}.$$ 
Here $2^{-j} \xi$ is the isotropic dilation of $\xi$ by $2^{-j}$. We then have 
$\sum_{j=0}^{\infty} p_j = 1.$
Now for $f \in \s$, $j \geq 0$, we define $$P_j f(x) = \int_{\mathbb{R}^N} p_j(\xi) \widehat{f}(\xi) e^{2\pi i x \cdot \xi} d\xi.$$ Then 
%as is well-known, each $P_j$ extends as a continuous operator from $L^p(\mathbb{R}^N)$ into itself, and 
for all $f \in L^p(\mathbb{R}^N)$, we have $f = \sum_{j=0}^{\infty} P_j f$ with convergence in $L^p(\mathbb{R}^N)$.
Now write $\tilde{P}_j = \sum_{|j'-j| \leq 1} P_{j'}$. Then 
%$P_j = \tilde{P}_j^* P_j$, so
$$f = \sum_{j=0}^{\infty} \tilde{P}_j^* P_j f.$$
% for all $f \in L^p(\mathbb{R}^N)$, where the convergence is in the $L^p$ sense. 
We also have the Littlewood-Paley inequality
$$\|f\|_{L^p} \simeq \left\| \left( \sum_{j=0}^{\infty} |P_j f|^2 \right)^{1/2} \right\|_{L^p}, \quad 1 < p < \infty.$$

The second version is a `non-isotropic' one. Let $\psi \in C^{\infty}_c(\mathbb{R}^N)$ such that 
$$ \psi(\xi)  = \begin{cases} 1 \quad \text{if $|\xi| \leq 1$} \\   0 \quad \text{if $|\xi| \geq 2$} \end{cases}.$$ 
Let 
$$
\psi_0(\xi) = \psi(\xi), \quad \psi_k(\xi) = \psi(2^{-k} \circ \xi) - \psi(2^{-(k-1)} \circ \xi) \quad \text{for $k \geq 1$}.
$$
Here $2^{-k} \circ \xi$ is the non-isotropic dilation of $\xi$, defined by $$2^{-k} \circ (\xi', \xi_N) = (2^{-k} \xi', 2^{-2k} \xi_N).$$ Now let
$$ %\begin{equation} \label{eq:qkdef}
q_k(x,\xi) = \psi_k(M_x \xi) \quad \text{for $k \geq 0$,}
$$ %\end{equation}
and define, for $f \in \s$,
$$Q_k f(x) = \int_{\mathbb{R}^N} q_k(x,\xi) \widehat{f}(\xi) e^{2\pi i x \cdot \xi} d\xi.$$ 
One can show that
\begin{equation} \label{eq:Lpequiv}
\left\| \left( \sum_{k=0}^{\infty} |Q_k f|^2 \right)^{1/2} \right\|_{L^p} \simeq_p \|f\|_{L^p}
\end{equation}
for all $f \in L^p(\mathbb{R}^N)$, $1 < p < \infty$.

Furthermore, one can show that given $1 < p < \infty$, there exists $R = R(p)$, such that if
$$\tilde{Q}_k = \sum_{|k'-k| \leq R} Q_{k'},$$ then there exists a pseudodifferential operator $E = E_p \colon L^p(\mathbb{R}^N) \to L^p(\mathbb{R}^N)$, with $I-E$ is invertible on $L^p(\mathbb{R}^N)$, such that
$$f = \sum_{k=0}^{\infty} \tilde{Q}_k^* Q_k (I-E)^{-1}  f$$ for all $f \in L^p(\mathbb{R}^N)$. Here the convergence in $L^p(\mathbb{R}^N)$.

\subsection{The strong maximal function}

We will also need to introduce three maximal functions. The first one is just the standard (isotropic) Hardy-Littlewood maximal function:
$$Mf(x) = \sup_{r > 0} \fint_{|y-x| < r} |f(y)| dy.$$
(Here $\fint$ denotes the average over the domain of integration.) The second one is the `non-isotropic' maximal function, adapted to the real-variable structure given by the quasi-distance $d(x,y)$:
$$M_{\D} f(x) = \sup_{s > 0} \fint_{d(y,x) < s} |f(y)| dy.$$
Finally, let $Q_{r,s}(x)$ be the `rectangular cube' 
\begin{align*}
Q_{r,s}(x) = & \left\{ y \in \mathbb{R}^N \colon \sum_{i=1}^{N-1} \left| \Theta_0(x,y)_i \right| \leq s, \left| \Theta_0(x,y)_N \right| \leq r \right\}.
\end{align*}
The last maximal function we need is a `strong' maximal function, defined by
$$\M_0 f(x) = \sup_{\substack{0 < r,\, s < 1 \\ s^2 \leq r \leq s}} \, \fint_{Q_{r,s}(x)} |f(y)| dy.$$
It is well-known that $M$ and $M_{\D}$ are weak-type (1,1) and strong type $(p,p)$ for all $1 < p \leq \infty$. On the other hand, 
\begin{equation} \label{eq:Mmixedbound}
\M_0 f(x) \lesssim M_{\D} M f(x).
\end{equation}
In fact, if $0 < r, s < 1$ and $s^2 \leq r \leq s$, then writing $\eta$ for the characteristic function of the Euclidean unit ball, we have
\begin{align*}
\fint_{Q_{s^2,s}(x)} \fint_{|y-z|<r} |f(y)| dy dz
&= \frac{1}{s^{N+1} r^N} \int \chi_{Q_{s^2,s}(x)}(z) \eta(\frac{z-y}{r}) |f(y)| dy dz \\
&\geq C \frac{1}{s^{N+1} r^N} \int_{Q_{r,s}(x)} \left( \int \eta(\frac{z-y}{r}) \chi_{Q_{s^2,s}(x)}(z) dz \right) |f(y)| dy \\
&= C \frac{1}{s^{N-1} r} \int_{Q_{r,s}(x)} |f(y)| dy \\
& \geq C \M_0 f(x).
\end{align*}
As a result, (\ref{eq:Mmixedbound}) follows. In particular, $\M_0$ is bounded on $L^p(\mathbb{R}^N)$ for all $1 < p \leq \infty$.

Note that $\M_0$ defined above considers only sup over small values of $r$ and $s$. In applications, it is often convenient to consider the global maximal function
$$
\M f(x) = \M_0 f(x) +  \sup_{r > 1} \fint_{|y-x| < r} |f(y)| dy + \sup_{s > 1} \fint_{d(y,x) < s} |f(y)| dy.
$$
Then
$$
M f \leq \M f, \quad M_{\D} f \leq \M f, \quad \text{and} \quad \M f \leq M_{\D} M f.
$$
In particular, $\M$ is bounded on $L^p(\mathbb{R}^N)$ for all $1 < p \leq \infty$.

\subsection{The $L^p$ boundedness}

We now proceed to prove Theorem~\ref{thm:Lptwoflag}. The main estimate we need in the proof of Theorem~\ref{thm:Lptwoflag} is the following:

\begin{lemma} \label{lem:twoflagmaxest}
Let $a \in S^{m,n}$. Then for $f \in L^p(\mathbb{R}^N)$, we have
$$
|T_a P_j^* Q_k^* f| \lesssim 2^{jm} 2^{kn} \M f
$$
almost everywhere.
\end{lemma}

One can see this by writing 
\begin{equation} \label{eq:TaPj*Qk*}
T_a P_j^* Q_k^* f(x) = \int_{\mathbb{R}^N} \int_{\mathbb{R}^N} a(x,\xi) p_j(\xi) q_k(y,\xi) f(y) e^{2\pi i (x-y) \cdot \xi} dy d\xi
\end{equation} 
and estimating the kernel of $T_a P_j^* Q_k^*$, namely 
$$
K_{j,k}(x,y) = \int_{\mathbb{R}^N} a(x,\xi) p_j(\xi) q_k(y,\xi) e^{2\pi i (x-y) \cdot \xi} d\xi.
$$
(Note that the use of $P_j^*$ and $Q_k^*$ (instead of their adjoints) allows for a simple expression of this kernel $K_{j,k}$.) In fact, the kernel $K_{j,k}(x,y)$ is only non-zero when $k-\ell_0 \leq j \leq 2k+\ell_0$ for some absolute constant $\ell_0$, and satisfies
$$
|K_{j,k}(x,y)| \lesssim 
\displaystyle{\frac{2^{km} 2^{jn} 2^{k(N-1)} 2^j}{ (1+2^k|x-y|+2^j|\Theta_0(x,y)_N|)^{-M}}}.
$$
We omit the details.

As a result of Lemma~\ref{lem:twoflagmaxest}, we have the following:

\begin{lemma} \label{lem:LpTacancel}
Suppose $a \in S^{0,0}$. Then for $f \in L^p(\mathbb{R}^N)$, we have
$$|Q_{k'} P_{j'} T_a \tilde{P}_j^* \tilde{Q}_k^* f| \lesssim_R 2^{-|j-j'|} 2^{-|k-k'|} \M f$$ 
almost everywhere.
\end{lemma}

\begin{proof}
Suppose for instance $j < j'$ and $k > k'$. Then we write $$Q_{k'} P_{j'} T_a = 2^{-j'} 2^{k'} (2^{-k'} Q_{k'})(2^{j'} P_{j'}) T_a = 2^{-j'} 2^{k'} T_b$$ for some $b \in S^{1,-1}$. Thus invoking the previous lemma, we have $$|Q_{k'} P_{j'} T_a \tilde{P}_j^* \tilde{Q}_k^* f| =  2^{-j'} 2^{k'} | T_b \tilde{P}_j^* \tilde{Q}_k^* f| \lesssim 2^{-j'+j} 2^{k'-k} \M f,$$ as desired. The other cases can be handled similarly.
\end{proof}

\begin{proof}[Proof of Theorem~\ref{thm:Lptwoflag}]
Assume without loss of generality that $f \in \s$. We want to show $\|T_a f\|_{L^p} \lesssim \|f\|_{L^p}$. Now
\begin{align*}
\|T_a f\|_{L^p} & \lesssim_p \left\| \left( \sum_{j=-\infty}^{\infty} \sum_{k=-\infty}^{\infty} |Q_k P_j T_a f|^2 \right)^{1/2} \right\|_{L^p}
\end{align*}
where we define $Q_k = P_j = 0$ if $j, k < 0$.
But
$$f = \sum_{k'=-\infty}^{\infty} \sum_{j'=-\infty}^{\infty} \tilde{P}_{j'}^* \tilde{Q}_{k'}^* Q_{k'} (I-E)^{-1} P_{j'}  f$$ with convergence in $L^p(\mathbb{R}^N)$. Thus
$$Q_k P_j T_a f = \sum_{k'} \sum_{j'} Q_k P_j T_a \tilde{P}_{j+j'}^* \tilde{Q}_{k+k'}^* Q_{k+k'} (I-E)^{-1} P_{j+j'} f.$$
It then follows from Lemma~\ref{lem:LpTacancel} that
$$
|Q_k P_j T_a f| \lesssim \sum_{k'} \sum_{j'}  2^{-|j'|-|k'|} \M Q_{k+k'} (I-E)^{-1} P_{j+j'} f.
$$
Hence
\begin{align*}
\|T_a f\|_{L^p}
& \lesssim_p \sum_{k'} \sum_{j'} 2^{-|j'|-|k'|} \left\| \left( \sum_{j} \sum_{k} |\M Q_{k+k'} (I-E)^{-1} P_{j+j'} f|^2 \right)^{1/2} \right\|_{L^p},
%& \lesssim_p \left\| \left( \sum_{j} \sum_{k} |M_{\D} M_I  Q_k (I-E)^{-1} P_j f|^2 \right)^{1/2} \right\|_{L^p},
\end{align*}
which is bounded by $C_p \|f\|_{L^p}$.
\end{proof}

\section{Two special ideals in $S^{0,0}$}

Within $S^{0,0}$ there are two special ideals of symbols, namely $S^{\ep,-2\ep}$ and $S^{-\ep,\ep}$, $\ep > 0$, that enjoy better properties than symbols in $S^{0,0}$.

\subsection{Weak-type (1,1) estimates}

Theorem~\ref{thm:Lptwoflag} shows that operators of class $S^{0,0}$ are bounded on $L^p$ for $1 < p < \infty$. On the other hand, in general, an operator of class $S^{0,0}$ is not of weak-type $(1,1)$. This can be seen, for instance, by considering the operators $S_z$ in the Section~\ref{sect:smoothing}, when $\text{Re}\, z = 0$. Nevertheless, the two ideals of $S^{0,0}$ mentioned above, namely $S^{\ep,-2\ep}$ and $S^{-\ep,\ep}$, $\ep > 0$, give rise to operators that are weak-type $(1,1)$, as is shown in the following theorem.

\begin{thm} \label{thm:Sep-2epweak11}
Suppose $a \in S^{\ep,-2\ep}$ or $S^{-\ep,\ep}$ for some $\ep > 0$. Then $T_a$ is of weak-type $(1,1)$.
\end{thm}
 
\begin{proof}
Suppose $a(x,\xi) \in S^{\ep,-2\ep}$ for some $\ep > 0$. The key is to prove that for some absolute constant $C_0 > 1$, the kernel $K(x,y)$ of $T_a$ satisfies, for any $y_1, y_2 \in \mathbb{R}^N$,
\begin{equation} \label{eq:weak111}
\int_{d(x,y_1) > C_0 d(y_1,y_2)} |K(x,y_1) - K(x,y_2)| dx \lesssim 1.
\end{equation}
Similarly, when $a \in S^{-\ep,\ep}$, the key is to show that
\begin{equation} \label{eq:weak113}
\int_{|x-y_1| > C_0 |y_1-y_2|} |K(x,y_1) - K(x,y_2)| dx \lesssim 1.
\end{equation}
These are consequences of the kernel estimates in Theorem~\ref{thm:KernelTheta}. We omit the proofs.
\end{proof}

\subsection{Preservation of H\"older spaces}

Let $\Lambda^{\alpha}$, $\alpha > 0$ be the ordinary `isotropic' Lipschitz space on $\mathbb{R}^N$. We will also need a `non-isotropic' Lipschitz space $\Gamma^{\alpha}$, which we define as follows.

For $0 < \alpha < 1$, we say $f \in \Gamma^{\alpha}$, if and only if $f \in L^{\infty}$, and $$|f(x)-f(y)| \leq C d(x,y)^{\alpha} \quad \text{for all $x,y \in \mathbb{R}^N$}.$$
We then define $$\|f\|_{\Gamma^{\alpha}} \simeq \|f\|_{L^{\infty}} + \sup_{x \ne y} \frac{|f(x)-f(y)|}{d(x,y)^{\alpha}}.$$

More generally, for $\alpha > 0$, let $s$ be the integer so that $\alpha \in [s,s+1)$. We say $f \in \Gamma^{\alpha}$, if $$\|f\|_{L^{\infty}} \leq C,$$ and for each $x \in \mathbb{R}^N$, $r > 0$, there exists a polynomial $P_{x,r}(y)$ of degree $\leq s$ such that 
\begin{equation} \label{eq:estbyTaylor}
\sup_{y \in B(x,r)} |f(y)-P_{x,r}(y)| \leq C r^{\alpha}.
\end{equation}
Here $B(x,r)$ is the non-isotropic ball of radius $r$. We write $\|f\|_{\Gamma^{\alpha}}$ for the least possible $C$ in the above inequalities.

We then have:

\begin{prop} \label{prop:equivnonisoLip2}
Suppose $\alpha > 0$. 
\begin{enumerate}[(a)]
\item  If $f \in \Gamma^{\alpha}$, then there exists a decomposition $$f = \sum_{k=0}^{\infty} f_k, \quad \text{with} \quad \|\partial_x^{\gamma} X'^{\lambda} f_k\|_{L^{\infty}} \leq C 2^{k(2|\gamma|+|\lambda|-\alpha)}$$
for all $k \geq 0$ and all $0 \leq |\gamma| + |\lambda| \leq \alpha + 1$. Here $C \lesssim \|f\|_{\Gamma^{\alpha}}$.
\item Conversely, if $f$ admits a decomposition as in part (a), then $f \in \Gamma^{\alpha}$ with $\|f\|_{\Gamma^{\alpha}} \lesssim C.$
\end{enumerate}
\end{prop}

See Campanato \cite{MR0167862}, Krantz \cite{MR556266}, and \cite[Section 9]{MR549321} for some relevant facts.

As is known, $S^0$ does not preserve $\Gamma^{\alpha}$ for $\alpha > 0$, and $S^0_{\mathcal{D}}$ does not preserve $\Lambda^{\alpha}$ for $\alpha > 0$. Nonetheless, we have the following theorem:

\begin{thm} \label{thm:twoflagHolder}
Let $a \in S^{\varepsilon,-2\varepsilon}$ or $S^{-\varepsilon,\varepsilon}$ for some $\varepsilon > 0$. Then 
\begin{enumerate}[(i)]
\item $T_a \colon \Lambda^{\alpha} \to \Lambda^{\alpha}$ for all $\alpha > 0$; and
\item \label{thm:twoflagHoldernoniso} $T_a \colon \Gamma^{\alpha} \to \Gamma^{\alpha}$ for all $\alpha > 0$.
\end{enumerate}
\end{thm}

We will only prove part (\ref{thm:twoflagHoldernoniso}) of the theorem, since the proof of the other part is similar. The key are the following two lemma (whose proofs we omit):
\begin{lemma} \label{lem:idealLinfty}
Suppose $a \in S^{\varepsilon,-2\varepsilon}$ or $S^{-\varepsilon,\varepsilon}$ for some $\varepsilon > 0$. Then
$$\|\partial_x^{\gamma} X'^{\lambda} T_a Q_k^*\|_{L^{\infty} \to L^{\infty}} \leq C 2^{k(2|\gamma|+|\lambda|)}$$
and for all $M \geq 0$, we have
$$\|\partial_x^{\gamma} X'^{\lambda} T_a Q_k^* Q_l^*\|_{L^{\infty} \to L^{\infty}} \leq C_M 2^{-M(k-l)} 2^{l(2|\gamma|+|\lambda|)} \quad \text{if $k \geq l$}.$$
\end{lemma}

\begin{lemma} \label{lem:QkGammagamma}
Suppose $f \in \Gamma^{\alpha}$ for some $\alpha > 0$ with $\|f\|_{\Gamma^{\alpha}} \leq 1$. Then
$$\|Q_m f \|_{L^{\infty}} \leq C 2^{-m \alpha}$$
and for all $M \geq 0$, we have
$$\|Q_l^* Q_m f\|_{L^{\infty}} \leq C_M 2^{-M(l-m)} 2^{-m\alpha}  \quad \text{if $l \geq m$}.$$
\end{lemma}

\begin{proof}[Proof of Theorem~\ref{thm:twoflagHolder} (\ref{thm:twoflagHoldernoniso})] 
Suppose $a \in S^{\varepsilon,-2\varepsilon}$ or $S^{-\varepsilon,\varepsilon}$ for some $\varepsilon > 0$. Then 
$$T_a f = \sum_{k = 0}^{\infty} F_k, \quad \text{where} \quad F_k = T_a Q_k^* f.$$ 
We want to show that 
$$
\|\partial_x^{\gamma} X'^{\lambda} F_k\|_{L^{\infty}}\leq C 2^{k(2|\gamma|+|\lambda|-\alpha)}.
$$
But
$$
\|\partial_x^{\gamma} X'^{\lambda} F_k\|_{L^{\infty}} \leq \sum_{m \geq 0} \|\partial_x^{\gamma} X'^{\lambda} T_a Q_k^* Q_m f\|_{L^{\infty}} = \sum_{0 \leq m \leq k} + \sum_{m > k}.
$$
The sum over $m > k$ can be estimated by 
$$\sum_{m > k} C 2^{k(2|\gamma|+|\lambda|)} 2^{-m\alpha} = C  2^{k(2|\gamma|+|\lambda|)} 2^{-k\alpha}.$$
In order to take the sum over $0 \leq m \leq k$, let's fix one such $m$. Then
$$
\|\partial_x^{\gamma} X'^{\lambda} T_a Q_k^* Q_m f\|_{L^{\infty}} \leq \sum_{0 \leq l \leq \frac{k+m}{2}} + \sum_{l > \frac{k+m}{2}} \|\partial_x^{\gamma} X'^{\lambda} T_a Q_k^* Q_l^* Q_m f\|_{L^{\infty}} = I + II.
$$
The first sum is estimated by
\begin{align*}
|I| 
&\leq \sum_{0 \leq l \leq \frac{k+m}{2}} \|\partial_x^{\gamma} X'^{\lambda} T_a Q_k^* Q_l^*\|_{L^{\infty} \to L^{\infty}} \|Q_m f\|_{L^{\infty}} \\
& \leq \sum_{0 \leq l \leq \frac{k+m}{2}} C_M 2^{-M(k-l)} 2^{l(2|\gamma|+|\lambda|)} 2^{-m \alpha} \\
& \leq C_M 2^{-\frac{M}{2}(k-m)} 2^{k(2 |\gamma| + |\lambda|)} 2^{-m \alpha}
\end{align*}
for any $M \geq 0$. Next, the second sum is bounded by 
\begin{align*}
|II| 
&\leq \sum_{l > \frac{k+m}{2}} \|\partial_x^{\gamma} X'^{\lambda} T_a Q_k^*\|_{L^{\infty} \to L^{\infty}} \|Q_l^* Q_m f\|_{L^{\infty}} \\
& \leq \sum_{l > \frac{k+m}{2}} C_M 2^{k(2|\gamma|+|\lambda|)} 2^{-M(l-m)} 2^{-m \alpha} \\
& = C_M 2^{k(2 |\gamma| + |\lambda|)} 2^{-\frac{M}{2}(k-m)} 2^{-m \alpha} 
\end{align*}
which is the same bound as we have obtained in $I$. Now pick $M$ such that $M/2 > \alpha$. Then we can sum this over all $0 \leq m \leq k$, and bound this by $C 2^{k(2 |\gamma| + |\lambda|)} 2^{-k \alpha}$ as desired. 
\end{proof}

\section{Smoothing properties in $L^p$} \label{sect:smoothing}

\begin{thm} \label{thm:twoflagsmoothing}
Let $a \in S^{m,n}$ for some 
$$-1 < m < 0, \quad -(N-1) < n < 0.$$ For $p \geq 1$, define
$$\frac{1}{p^*} = \frac{1}{p} - \gamma, \quad \gamma := \min\left\{ \frac{|m+n|}{N}, \frac{|2m+n|}{N+1} \right\},$$ if $1/p > \gamma$.  
Then:
\begin{enumerate}[(i)]
\item $T_a \colon L^p \to L^{p^*}$ whenever $1 < p \leq p^* < \infty$;
\item \label{item:Thmtwoflagsmoothing_b} If in addition $$\frac{m+n}{N} \neq \frac{2m+n}{N+1},$$ i.e. if $n \ne m(N-1)$, then $T_a$ is weak-type $(1,1^*)$.
\end{enumerate}
\end{thm}

This easily implies that such $T_a$ maps $L^p$ to $L^q$ whenever $p \leq q < p^*$, $p \geq 1$. One can also show that these results are sharp, by homogeneity considerations about the classes $S^{m,0}$ and $S^{0,n}$.

Note that Theorems~\ref{thm:Lptwoflag} and \ref{thm:Sep-2epweak11} can be thought of as a version of the present theorem in the limiting case $\gamma = 0$. Furthermore, the result in the present theorem is stronger than the estimate predicted by the equation
$$\frac{1}{p^*} = \frac{1}{p} - \frac{|m|}{N} - \frac{|n|}{N+1},$$ which would be the one obtained if $S^{m,n}$ smoothes only like $S^{m,0} \circ S^{0,n}$. This can be seen in the figure below: suppose $P = (m,n)$ is on the `critical' dotted line, where $m = -\gamma$ and $n = -(N-1)\gamma$ for some $\gamma \in (0,1)$. Suppose further that $A = (-N\gamma,0)$, $B = (0,-(N+1)\gamma)$. Let $1/p^* = 1/p - \gamma$, $p > 1$. If $a \in S^{m,n}$ with $(m,n)$ on the solid lines, then Theorem~\ref{thm:twoflagsmoothing} says that $T_a$ maps $L^p$ to $L^{p^*}$. If $S^{m,n}$ were smoothing only like $S^{m,0} \circ S^{0,n}$, then only those symbols on the dashed line map $L^p$ to $L^{p^*}$.

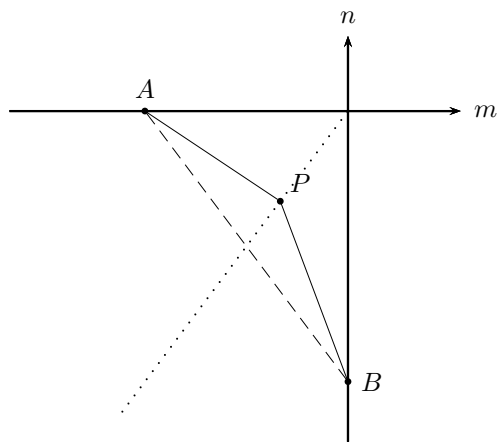
\begin{center}
\psset{xunit=3cm}
\psset{yunit=2cm}
\begin{pspicture}(-1.5,-2.5)(0.5,1)
\psaxes[labels=none,ticks=none]{->}(0,0)(-1.5,-2.2)(0.5,0.5) 
\psline[linestyle=dotted]{-}(0,0)(-1,-2)
\psline[linestyle=solid,linewidth=0.01]{-*}(-0.9,0)(-0.3,-0.6)
\psline[linestyle=solid,linewidth=0.01](-0.3,-0.6)(0,-1.8)
\psline[linestyle=dashed,linewidth=0.01]{*-*}(-0.9,0)(0,-1.8)
\uput[0](0.5,0){$m$}
\uput[90](0,0.5){$n$}
\uput[90](-0.9,0){$A$}
\uput[0](0,-1.8){$B$}
\uput[45](-0.3,-0.6){$P$}
\end{pspicture}
\end{center}

To prove the part (i) of Theorem~\ref{thm:twoflagsmoothing}, by the inclusion relations (\ref{eq:symbol_inclusion}), one only needs to consider the case when $(m,n)$ is on the critical line, i.e. when $m=-\gamma$, $n=-(N-1)\gamma$, for some $\gamma \in (0,1)$. Let 
$$K_z(u) = z^2 \chi\left( \frac{|u''|}{|u'|} \right) \chi \left( \frac{|u'|^2}{|u''|} \right) \frac{1}{|u'|^{(N-1)(1-z)} |u''|^{1-z}}$$ where $\chi \in C^{\infty}_c([1/4,4])$ and is identically 1 on $[1/2,2]$. Then $$S_z f(x) = \int K_z(\Theta_0(x,y)) f(y) dy$$ is an analytic family of operators, and maps $L^1$ to $L^{\infty}$ when $\textrm{Re}\, z = 1$. Furthermore, $S_z$ maps $L^q$ to $L^q$ for all $1 < q < \infty$ when $\textrm{Re}\, z = 0$, with a bound that grows polynomially in $z$; this is because it can be shown, by Theorem~\ref{thm:KernelS00}, that $S_{iy}$ arise as pseudodifferential operators with symbols in $S^{0,0}$. Thus $S_{\gamma}$ maps $L^p$ to $L^{p^*}$ if $0 < \gamma < 1$, $p > 1$ and $p^* < \infty$. Now given $a \in S^{m,n}$ as in the statement of the theorem, one represents $T_a f(x)$ as $\int_{\mathbb{R}^N} f(y) k_0(x,\Theta_0(x,y)) dy$ as in Theorem~\ref{thm:KernelS00}. Then one can split $k_0$ into 3 parts, according to whether $|u'| < |u''|$, $|u'| > |u''|^{1/2}$, or $|u''| \leq |u'| \leq |u''|^{1/2}$. The contribution of the last part is bounded by $S_{\gamma}$. The rest can be bounded by purely isotropic or non-isotropic fractional integrals. Hence we are done.

To prove part (ii) of Theorem~\ref{thm:twoflagsmoothing}, since now $1^* \in (1,\infty)$, so that weak-$L^{1^*}$ is a normed space, one only need to show

\begin{prop} \label{prop:kernelweak1*} 
If $a \in S^{m,n}$, where $m$, $n$ satisfies the assumption in part (\ref{item:Thmtwoflagsmoothing_b}) of Theorem~\ref{thm:twoflagsmoothing}, then the kernel $K(x,y)$ of $T_a$ satisfies
$$
\sup_{y \in \mathbb{R}^N} \|K(x,y)\|_{L^{1^*,\infty}(dx)} \leq C.
$$
\end{prop}

We omit the proof.
 
\section{The operators on a compact manifold}

Let $M$ be a smooth manifold of real dimension $N$, and suppose a distribution $\mathcal{D}$ of codimension 1 tangent subspaces of $M$ is given. For simplicity, we will assume that $M$ is compact. We will now construct an algebra of pseudodifferential operators on $M$, that is adapted to the distribution $\mathcal{D}$, and that has mixed homogeneities. We will then see that most theorems in the previous sections continue to hold in our present context.

First, given any point $p$ on $M$, there exists a contractible open set $U$ containing $p$, a coordinate chart $x \colon U \simeq B(0,1) \subset \mathbb{R}^N$, and a frame of tangent vectors $X_1, \dots, X_N$ on $U$, such that $\mathcal{D}$ is spanned by $X_1, \dots, X_{N-1}$ at every point in $U$. Such a coordinate system is said to be an admissible coordinate system on $M$. We identify $U$ with an open subset of $\mathbb{R}^N$ via such a  coordinate chart $x$, and transplant the distribution $\mathcal{D}$ from $U$ onto this open subset (which we will still denote by $\mathcal{D}$ by abuse of notation). We extend this transplanted $\mathcal{D}$ into all of $\mathbb{R}^N$ as in our discussion in Section~\ref{subsect:Theta}.
A linear operator $S \colon C^{\infty}(M) \to C^{\infty}(M)$ is said to be a pseudodifferential operator of order $(m,n)$ adapted to $\mathcal{D}$, written $S \in \Psi^{m,n}(\mathcal{D})$, if the following holds:
\begin{enumerate}[(a)]
\item For any admissible coordinate chart $x \colon U \to \mathbb{R}^N$ and any $\chi_1, \chi_2 \in C^{\infty}_c(U)$, the operator $\chi_1 S \chi_2$ is given by $T_a$ for some $a \in S^{m,n}(\mathcal{D})$ in the coordinate system $x$.
\item For any $\psi_1$, $\psi_2 \in C^{\infty}_c(M)$ with disjoint support, there exists a smooth kernel $k(x,y) \in C^{\infty}(M \times M)$, such that
$$\psi_1 S \psi_2 f(x) = \int_M k(x,y) f(y) dy$$
for all $f \in C^{\infty}(M)$.
\end{enumerate}
We remark that the class of operators $\Psi^{m,n}(\mathcal{D})$ is well-defined, and invariant under diffeomorphisms. We will write $\Psi^{m,n}$ for $\Psi^{m,n}(\mathcal{D})$ when there is no confusion about the distribution $\mathcal{D}$ that is given. Again we can define an isotropic class of pseudodifferential operators, which we denote by $\Psi^m(\mathcal{D})$, and a non-isotropic class of pseudodifferential operators, which we denote by $\Psi^n_{\mathcal{D}}$.

Next, there is a counterpart, in our present context, of many of the previous theorems. For example,

\begin{thm} \label{thm:compose2}
If $S_1 \in \Psi^{m_1,n_1}$ and $S_2 \in \Psi^{m_2,n_2}$, then $S_1 \circ S_2 \in \Psi^{m,n},$ where $m = m_1 + m_2$, $n = n_1 + n_2$.
\end{thm}

For the next set of results, we need to pick a smooth volume form, and define some function spaces on $M$. To begin with, we define $L^p(M)$ and weak-$L^p(M)$ with respect to any smooth volume form on $M$. Moreover, we define $\Lambda^{\alpha}(M)$ to be the set of all functions $f$, such that $\chi f \in \Lambda^{\alpha}$ on $\mathbb{R}^N$ whenever $\chi$ is a smooth cut-off supported in an admissible coordinate chart. Similarly we define $\Gamma^{\alpha}(M)$.

\begin{thm} \label{thm:adjoint2}
Let $S \in \Psi^{m,n}$. Then the adjoint $S^*$ of $S$ with respect to any smooth measure on $M$ is in $\Psi^{m,n}$.
\end{thm}

\begin{thm} \label{thm:Lptwoflag2}
If $S \in \Psi^{0,0}$, then $S$ preserves $L^p(M)$, for all $1 < p < \infty$.
\end{thm}

\begin{thm} \label{thm:Lptwoflagideal2}
Suppose $S \in \Psi^{\ep,-2\ep}$ or $\Psi^{-\ep,\ep}$ for some $\ep > 0$. Then:
\begin{enumerate}[(i)]
\item $S$ is of weak-type $(1,1)$; and
\item $S$ preserves $\Lambda^{\alpha}(M)$ and $\Gamma^{\alpha}(M)$ for all $\alpha > 0$.
\end{enumerate}
\end{thm}

Furthermore, we have:

\begin{thm} \label{thm:Lptwoflagsmoothing2}
Suppose $-1 < m < 0$, $-(N-1) < n < 0$. For $p \geq 1$, define $p^*$ as in Theorem~\ref{thm:twoflagsmoothing}.
Then for $S \in \Psi^{m,n}$, we have:
\begin{enumerate}[(i)]
\item $S \colon L^p(M) \to L^{q}(M)$ whenever $p > 1$ and $q \leq p^* < \infty$;
\item If in addition $n \ne m(N-1)$, then $S$ is weak-type $(1,1^*)$.
\end{enumerate}
\end{thm}

\section{Applications}

Suppose now $M$ is the boundary of a smoothly bounded strongly pseudoconvex domain $\Omega$ in $\mathbb{C}^{d}$, $d \geq 2$. Then there is a natural distribution $\mathcal{D}$ of tangent subspaces on $M$, namely those spanned by the real and imaginary parts of the $(1,0)$ vectors that are tangent to $M$. 

One can then show that the relative solution operator $N$ of $\Box_b$ is an operator in the class $\Psi_{\mathcal{D}}^{-2}$. This holds because near the diagonal of $M \times M$, the kernel of $N$ is, up to better error terms, of the form
$N_0(\Theta(x,y))$, where $N_0$ is the relative solution operator of the standard $\Box_b$ on the Heisenberg group, and $\Theta(x,y)$ is defined locally by the exponential map as in (\ref{eq:Thetaexp}); see \cite{MR0367477}. One can then invoke an analog of Theorem~\ref{thm:KernelThetaEasyConverse} for $S^n_{\mathcal{D}}$ in place of $S^{m,n}$ in order to conclude the argument.

Next, it can be shown that the Szeg\"o projection $\textbf{S}$ on $M$ is in $\Psi^0_{\mathcal{D}}$, but one can say more about it: it is also in $\Psi^{\varepsilon,-2\varepsilon}$ for all $\varepsilon > 0$. In fact, since $\textbf{S}$ is a projection, i.e. $\textbf{S} = \textbf{S}^k$ for all $k$, by Theorem~\ref{thm:compose2}, it suffices to show that $\textbf{S} \in \Psi^{1,-2}$. But in the terminology of \cite{MR0450623}, the kernel $\textbf{S}(x,y)$ of $\textbf{S}$ is a kernel of weight $0$. Thus by Lemma 2 in that paper, there exists a kernel $K(x,y)$ of weight $1$, such that $\textbf{S}(x,y) = T_x K(x,y)$ for some vector field $T$ that is transverse to $\mathcal{D}$. ($T_x$ indicates that the derivative is in the $x$-variable.) Now a kernel of weight $1$ is an operator of class $\Psi^{-2}_{\mathcal{D}}$; this follows again from an analog of Theorem~\ref{thm:KernelThetaEasyConverse} for $S^n_{\mathcal{D}}$. It follows that $\textbf{S} \in \Psi^{1,-2}$ as desired.

Finally, in solving the $\overline{\partial}$-Neumann problem on $\Omega$, one is led to invert a Dirichlet-to-$\overline{\partial}$-Neumann operator $\Box^+$. As can be shown using \cite{MR0499319}, on p.110 when $d \geq 3$, and on p.118 when $d = 2$, and also using Propositions 3.2 through 3.5 of \cite{MR1194003}, there is a parametrix $A \in \Psi^{1,-2}$ such that
$$
\begin{cases}
\Box^+ A = I + E, \\
A \Box^+ = I + E'
\end{cases}
$$
where $E, E' \in \Psi^{-\infty}$. 

\begin{center}
\textbf{Acknowledgments}
\end{center}

Elias M. Stein was supported in part by the National Science Foundation (DMS-
0901040). Po-Lam Yung was supported in part by the National Science Foundation
(DMS-1201474).

\end{document}